\documentclass[12pt]{amsart}
\usepackage{graphicx}
\usepackage[utf8]{inputenc}
\usepackage{amsmath}
\usepackage{amssymb}
\usepackage{xcolor}
\usepackage{hyperref}
\usepackage{cleveref}
\usepackage{pgfplots}
\usepackage{float}
\pgfplotsset{compat=1.15}
\usepackage{mathrsfs}
\usetikzlibrary{arrows} 
\usepackage[margin=2.5cm]{geometry}

\newcommand{\vv}{\mathrm{v}}
\newcommand{\p}{\mathfrak{p}}
\newcommand{\q}{\mathfrak{q}}
\newcommand{\ass}{\mathrm{Ass}}

\newtheorem{theorem}{Theorem}[section]

\newtheorem{definition}[theorem]{Definition}

\newtheorem{remark}[theorem]{Remark}

\newtheorem{corollary}[theorem]{Corollary}
\newtheorem{property}[theorem]{Property}

\usepackage[backend=bibtex]{biblatex}
\title{Binomial expansion and the $\mathrm{v}$-number}

\author{Kamalesh Saha}
\address{Chennai Mathematical Institute, Siruseri, Tamil Nadu, India - 603103.}
\email{ksaha@cmi.ac.in; kamalesh.saha44@gmail.com}


\begin{document}

\begin{abstract}
Let $I\subset A$ and $J\subset B$ be two monomial ideals, where $A$ and $B$ are two polynomial rings with disjoint variables. Considering a general set-up of monomial filtrations, we study the behaviour of the $\vv$-function under binomial expansion. As an application, we get an explicit formula of $\vv((I+J)^{(k)})$ in terms of $\vv(I^{(i)})$ and $\vv(J^{(j)})$, where $L^{(k)}$ denote the symbolic power of an ideal $L$. Furthermore, an analogous formula is extended for the $\vv$-function of integral closure of $(I+J)^k$.
\end{abstract}

\maketitle

\section{Introduction}

Let $R=\bigoplus_{d\geq 0}R_{d}$ be a standard graded polynomial ring over a field $K$ and $L\subset R$ be a proper graded ideal. We denote the set of associated primes of $L$ by $\ass(L)$. The invariant called \textit{$\vv$-number} of $L$, introduced by Cooper et al. in \cite{cstpv20}, denoted by $\mathrm{v}(L)$, is defined as
$$\mathrm{v}(L):=\min\{d\geq 0 \mid \exists\,\, f\in R_d \text{ and } \mathfrak{p}\in \ass(L) \text{ with } L:f=\mathfrak{p}\}.$$
For each $\mathfrak{p}\in \mathrm{Ass}(L)$, one can define the \textit{local} $\vv$-number of $L$ at $\mathfrak{p}$, denoted by $\mathrm{v}_{\p}(L)$ as $\mathrm{v}_{\p}(L):=\min\{d\geq 0 \mid \exists\,\, f\in R_d \text{ with } L:f=\mathfrak{p}\}$. Thus, $\vv(L)=\min\{\vv_{\mathfrak{p}}(L)\mid \mathfrak{p}\in\mathrm{Ass}(L)\}$. The notion of $\vv$-number was introduced to investigate the asymptotic behaviour of the minimum distance function of projective Reed-Muller-type codes (see \cite{cstpv20}). Presently, the study of the $\vv$-number is emerging as a significant trend in the theory of commutative algebra. Numerous inquiries have already been carried out to examine the properties and potential applications of this number (see \cite{fm24_vsum} and references therein for details).\par 

Let $A=K[x_1,\ldots, x_n]$ and $B=K[y_1,\ldots, y_m]$ be two polynomial rings with disjoint sets of variables over the same field $K$. Let $I\subset A$ and $J\subset B$ be two graded ideals and $I+J:=IS+JS$, where $S=A\otimes_{K} B$. Now, corresponding to a fixed operation on ideals (such as powers, symbolic powers, integral closure of powers), consider a graded $I$-filtration $\{I_k\}$ and a graded $J$-filtration $\{J_k\}$. Then, a natural study is to investigate algebraic invariants like regularity, depth, etc., of $\{(I+J)_k\}$ in terms of those invariants of $\{I_k\}$ and $\{J_k\}$. In this direction, researchers showed significant interest when $(I+J)_k$ satisfies the binomial expansion $(I+J)_{k}=\sum_{i+j=k}I_{i}J_{j}$. This type of binomial expansion holds in the case of ordinary power filtration, symbolic power filtration, and, under some conditions, integral closure of powers, etc. For all these scenarios, people investigated the behaviour of the regularity and depth functions on $\{(I+J)_{k}\}$ in terms of those functions on $\{I_k\}$ and $\{J_k\}$ (see \cite{htt16sumpower}, \cite{hntt20symbsum}, \cite{hjkn23symbsumgen}, \cite{bha23}).\par 

Recently, Ficarra and Macias Marques \cite{fm24_vsum} studied the $\vv$-function of $(I+J)^k$. If $I$ and $J$ are monomial ideals, they gave an exact formula of $\vv((I+J)^k)$ in terms of $\vv(I^i)$ and $\vv(J^j)$, where $1\leq i,j\leq k$. Due to some computational evidence, they expect that their formula can be extended for arbitrary graded ideals. However, this would be an extremely challenging problem because till now, we do not know about the additivity of $\vv$-numbers, i.e. whether $\vv(I+J)=\vv(I)+\vv(J)$ holds always true or not. In \cite{ssvmon22}, the authors showed the additivity of $\vv$-numbers for monomial ideals, and in \cite{ambhore2023v}, the additivity of $\vv$-numbers has been shown for some other classes of graded ideals.\par 

In this paper, we give an explicit formula of the $\vv$-function of $(I+J)_k$ in terms of the $\vv$-function of $I_i$ and $J_j$, when $\{I_k\}$ and $\{J_k\}$ are filtrations of monomial ideals and $(I+J)_k$ satisfies the binomial expansion (see \Cref{thmmain}). The approach used to prove our main result resembles that of \cite[Theorem 4.1]{fm24_vsum}, where we assume certain appropriate conditions naturally fulfilled by ordinary power filtration. As a consequence, we derive a formula for $\vv(I+J)^{(k)}$ in terms of $\vv(I^{(i)})$ and $\vv(I^{(j)})$ as follows:
\medskip

\noindent \textbf{Corollary \ref{corsympower}.} Let $I\subset A$ and $J\subset B$ be two monomial ideals. If $\p\subset A$ and $\q\subset B$ be two monomial prime ideals such that $\p+\q\in\ass(I+J)^{(k)}$, then 
    $$\vv_{\p+\q}((I+J)^{(k)})=\min_{\substack{0\leq d<k \\ 
    \p\in\ass(I^{(k-d)}) \\ \q\in\ass(J^{(d+1)})}}(\vv_{\p}(I^{(k-d)})+\vv_{\q}(J^{(d+1)})).$$
    In particular, we have 
    $$\vv((I+J)^{(k)})= \min_{\substack{0\leq d<k \\ 
    \p\in\ass(I^{(k-d)}) \\ \q\in\ass(J^{(d+1)}) \\ \p+\q\in\ass((I+J)^{(k)})}}(\vv_{\p}(I^{(k-d)})+\vv_{\q}(J^{(d+1)})).$$
\medskip

Let $\overline{L}$ denote the integral closure of an ideal $L\subset R$. Then $\{\overline{L^k}\}$ forms a graded $L$-filtration of $R$. In \Cref{corintclosure}, we additionally establish that an analogous formula, akin to those observed in cases of ordinary and symbolic power filtrations, holds true for $\vv(\overline{(I+J)^k})$ whenever $\overline{(I+J)^k}$ satisfies the binomial expansion. For instance, when (i) $I$ is square-free and $\overline{I^{k}}=I^{(k)}$ for all $k\geq 1$, or (ii) $I$ is a normally torsion-free square-free monomial ideal.

\section{The $\vv$-number of binomial expansion}

Now, we will investigate the behaviour of the $\vv$-number with respect to the binomial expansion in the case of monomial ideals. Let us start by exploring monomial ideals and some of their properties.

\subsection{Monomial ideals} A monomial in a polynomial ring $R=K[z_1,\ldots,z_t]$ is a polynomial of the form $z_{1}^{c_{1}}\cdots z_{t}^{c_{t}}$ with each $c_{i}\in \mathbb{N}$, where $\mathbb{N}$ denotes the set of all non-negative integers. For a vector $\mathbf{c}=(c_1,\ldots,c_t)\in\mathbb{N}^t$, we write $\mathbf{z^c}$ to denote the monomial $z_{1}^{c_{1}}\cdots z_{t}^{c_{t}}$, i.e. $\mathbf{z^c}=z_{1}^{c_{1}}\cdots z_{t}^{c_{t}}$. An ideal $L\subset R$ is called a \textit{monomial ideal} if it is minimally generated by a set of monomials in $R$. The set of minimal monomial generators of $L$ is unique and denoted by $\mathcal{G}(L)$. If $\mathcal{G}(L)$ consists of only square-free monomials (i.e., if each $c_{i}\in\{0,1\}$), then we say $L$ is a \textit{square-free} monomial ideal. \par 

Let $L$ and $L'$ be two ideals of $R$. Then $L:L'= \{u\in R \mid uv\in L\,\,\text{for\,\, all}\,\,v\in L'\}$ is an ideal of $R$, known as the \textit{colon  ideal} of $L$ with respect to $L'$. 
For an element $f\in R$, we write $L:f$ to mean the ideal $L:\langle f\rangle$. If $L$ is a monomial ideal and $f\in R$ is a monomial, then it is well known that 
$$L:f=\big\langle\dfrac{u}{\mathrm{gcd}(u,f)}\mid u\in \mathcal{G}(L)\big\rangle.$$
The above expression of the colon ideal of a monomial ideal by a monomial plays a crucial role in establishing our results. Due to the above expression of colon ideals of monomial ideals, it is straightforward to verify the following two facts:
\medskip

\noindent \textbf{Fact 1.} Let $L_1,\ldots,L_k\subset R$ be monomial ideals and $f\in R$ be any monomial. Then $$(\sum_{i=1}^{k}L_i):f=\sum_{i=1}^{k}(L_i:f).$$

\noindent \textbf{Fact 2.} Let $A=K[x_1,\ldots,x_n]$ and $B=K[y_1,\ldots,y_m]$ be two polynomial rings with disjoint set of variables. Let $I\subset A$, $J\subset B$ be two monomial ideals and $\mathbf{x^a}\in A$, $\mathbf{y^b}\in B$ be any two monomials. Then 
$$IJ:\mathbf{x^ay^b}=(I:\mathbf{x^a})(J:\mathbf{y^b}).$$

\begin{definition}{\rm
    A \textit{graded filtration} of $R$ is a family $\mathcal{I} = \{I_{k}\}_{k\geq 0}$ of graded ideals of $R$ satisfying: (i) $I_{0} = R$; (ii) $I_{k+1} \subseteq I_{k}$ for all $k \geq 0$; (iii) $I_{k}I_{r} \subseteq I_{k+r}$ for all $k,r \geq 0$. Let $I$ be a graded ideal of $R$. A graded filtration $\mathcal{I}=\{I_{k}\}_{k\geq 0}$ is said to be a \textit{graded $I$-filtration} if $I^k\subseteq I_{k}$ for all $k\geq 1$. The ordinary powers, symbolic powers, and integral closure of powers of a graded ideal $I$ form $I$-filtrations of $R$.
}
\end{definition}

In this paper, we make use of filtrations, which satisfy the following property.

\begin{property}\label{property}{\rm
    Corresponding to a monomial ideal $I$ in a polynomial ring $R$, we consider a graded $I$-filtration $\mathcal{I}=\{I_{k}\}_{k\geq 0}$ with the following properties:
\begin{enumerate}
    \item[(a)] Each $I_{k}$ is a monomial ideal;
    \item[(b)] If $\p\in \ass(I_{k})$ and $f\in R$ be a monomial such that $I_k:f=\p$, then $f\in I_{k-1}$.
\end{enumerate}
}
\end{property}

\begin{theorem}\label{thmmain}
    Let $I\subset A$ and $J\subset B$ be two monomial ideals such that $(I+J)_{k}=\sum_{i+j=k}I_{i}J_{j}$, where the filtration $\{I_i\}$ and $\{J_j\}$ satisfy \Cref{property}. If $\p\subset A$ and $\q\subset B$ be two monomial prime ideals such that $\p+\q\in\ass(I+J)_k$, then 
    $$\vv_{\p+\q}((I+J)_{k})=\min_{\substack{0\leq d<k \\ 
    \p\in\ass(I_{k-d}) \\ \q\in\ass(J_{d+1})}}(\vv_{\p}(I_{k-d})+\vv_{\q}(J_{d+1})).$$
    Moreover, we have 
    $$\vv((I+J)_{k})= \min_{\substack{0\leq d<k \\ 
    \p\in\ass(I_{k-d}) \\ \q\in\ass(J_{d+1}) \\ \p+\q\in\ass((I+J)_{k})}}(\vv_{\p}(I_{k-d})+\vv_{\q}(J_{d+1})).$$
\end{theorem}

\begin{proof}
Without loss of generality, let us assume $\p=\langle x_1,\ldots,x_r\rangle$ and $\q=\langle y_1,\ldots,y_s\rangle$, where $1\leq r\leq n$ and $1\leq s\leq m$. Since $(I+J)_{k}$ is a monomial ideal, there exists a monomial $f=\mathbf{x^ay^b}\in S=A\otimes_{K}B$ such that $(I+J)_{k}:f=\p+\q$ and $\deg(f)=\vv_{\p+\q}((I+J)_{k})$. Now, due to the binomial expansion and Fact 1, we have the following:
\begin{align*}
    (I+J)_k:f &=(\sum_{i+j=k}I_{i}J_{j}):\mathbf{x^ay^b} \\ &=(\sum_{h=0}^{k}I_{k-h}J_{h}):\mathbf{x^ay^b} \\ 
    &=\sum_{h=0}^{k}(I_{k-h}J_{h}:\mathbf{x^ay^b}).
\end{align*}
By Fact 2, the above equality simplifies to
\begin{align}\label{eq1}
    \p+\q=(I+J)_k:f=\sum_{h=0}^{k}(I_{k-h}:\mathbf{x^a})(J_{h}:\mathbf{y^b}).
\end{align}
Therefore, for each $1\leq i\leq r$, there exists $d_{i}\in\{0,\ldots,k\}$ such that 
$$x_{i}\in (I_{k-d_i}:\mathbf{x^a})(J_{d_i}:\mathbf{y^b}).$$
Now, choose $d=\max\{d_1,\ldots,d_r\}$. Then there exists $x\in\{x_1,\ldots,x_r\}$ such that $x\in (I_{k-d}:\mathbf{x^a})(J_{d}:\mathbf{y^b})$. Since $x$ is a variable in $A$, we should have $x\in I_{k-d}:\mathbf{x^a}$ and $J_{d}:\mathbf{y^b}=S$, i.e. $\mathbf{y^b}\in J_d$. Since $J_{d}\subseteq J_{d-1}\subseteq \cdots \subseteq J_{0}$, we have $\mathbf{y^b}\in J_{j}$ for all $0\leq j\leq d$. We claim that $d<k$. Suppose $d=k$. Then $(I_{k-d}:\mathbf{x^a})(J_d:\mathbf{y^b})=S$, and thus, by the equality \eqref{eq1}, we have $(I+J)_{k}:f=S$ lead to a contradiction. Therefore, $d<k$. Now, for each $1\leq i\leq r$, we have $x_{i}\in I_{k-d_i}:\mathbf{x^a}\subseteq I_{k-d}:\mathbf{x^a}$  as $d_i\leq d$. Thus, it is clear that $\p\subseteq I_{k-d}:\mathbf{x^a}$. Due to the equality \eqref{eq1} and $\mathbf{y^b}\in J_{d}$, we have $(I_{k-d}:\mathbf{x^a})(J_{d}:\mathbf{y^b})=I_{k-d}:\mathbf{x^a}\subseteq \p+\q$. Since the minimal monomial generators of $I_{k-d}:\mathbf{x^a}$ belong to the ring $A$, we must have $I_{k-d}:\mathbf{x^a}=\p$, and hence,
\begin{align}\label{eq2}
    \vv_{\p}(I_{k-d})\leq \deg(\mathbf{x^a}).
\end{align}
Since $I_{k-d}:\mathbf{x^a}=\p$, we have $\mathbf{x^a}\not\in I_{k-d}$, and thus, by the choice of our filtration, $\mathbf{x^a}\in I_{k-(d+1)}$. Therefore, $\mathbf{x^a}\in I_{k-h}$, i.e. $I_{k-h}:\mathbf{x^a}=S$ for all $d+1\leq h\leq k$ as $I_{k-(d+1)}\subseteq\cdots\subseteq I_{1}\subseteq I_{0}$. In a similar way, $\mathbf{y^b}\in J_{d}$ implies $J_{h}:\mathbf{y^b}=S$ for all $1\leq h\leq d$. Hence, the equation \eqref{eq1} becomes 
\begin{align}\label{eq3}
   \p+\q=(I+J)_{k}:f=\sum_{h=0}^{d}(I_{k-h}:\mathbf{x^a})+\sum_{h=d+1}^k(J_{h}:\mathbf{y^b}). 
\end{align}
Again, due to $I_{k}\subseteq \cdots\subseteq I_{k-d}$ and $J_{d+1}\supseteq \cdots\supseteq J_{k}$, we have $I_{k}:\mathbf{x^a}\subseteq \cdots\subseteq I_{k-d}:\mathbf{x^a}$ and $J_{d+1}:\mathbf{y^b}\supseteq \cdots\supseteq J_{k}:\mathbf{y^b}$. Thus, equation \eqref{eq3} further reduces to
\begin{align}\label{eq4}
    \p+\q=(I+J)_{k}:f=(I_{k-d}:\mathbf{x^a})+(J_{d+1}:\mathbf{y^b}). 
\end{align}
Since $I_{k-d}:\mathbf{x^a}=\p$ and the minimal monomial generators of $J_{d+1}:\mathbf{y^b}$ belong to $B$, it follows from \eqref{eq4} that $J_{d+1}:\mathbf{y^b}=\q$. Hence,
\begin{align}\label{eq5}
    \vv_{\q}(J_{d+1})\leq \deg(\mathbf{y^b}).
\end{align}
Combining inequalities \eqref{eq2} and \eqref{eq5}, we get
$$\vv_{\p+\q}((I+J)_{k})=\deg(f)=\deg(\mathbf{x^a})+\deg(\mathbf{y^b})\geq \vv_{\p}(I_{k-d})+\vv_{\q}(J_{d+1}).$$
Therefore, 
\begin{align}\label{eq6}
    \vv_{\p+\q}((I+J)_{k})\geq\min_{\substack{0\leq d<k \\ 
    \p\in\ass(I_{k-d}) \\ \q\in\ass(J_{d+1})}}(\vv_{\p}(I_{k-d})+\vv_{\q}(J_{d+1})).
\end{align}
\par 

Now, to establish the reverse inequality, choose $d\in\{0,1,\ldots,k-1\}$, $\p\in \ass(I_{k-d})$ and $\q\in\ass(J_{d+1})$. Then there exist two monomials $\mathbf{x^a}\in A$ and $\mathbf{y^b}\in B$ such that $I_{k-d}:\mathbf{x^a}=\p$, $J_{d+1}:\mathbf{y^b}=\q$ with $\vv_{\p}(I_{k-d})=\deg(\mathbf{x^a})$ and $\vv_{\q}(J_{d+1})=\deg(\mathbf{y^b})$. By the choice of our filtration, $I_{k-h}:\mathbf{x^a}=S$ for all $d+1\leq h\leq k$ and $J_{h}:\mathbf{y^b}=S$ for all $1\leq h\leq d$. Then due to equation \eqref{eq1}, we get
$$(I+J)_{k}:\mathbf{x^ay^b}=\sum_{h=0}^{d}(I_{k-h}:\mathbf{x^a})+\sum_{h=d+1}^k(J_{h}:\mathbf{y^b}).$$
Similar argument as before $I_{k}\subseteq \cdots\subseteq I_{k-d}$ and $J_{d+1}\supseteq \cdots\supseteq J_{k}$ imply $I_{k}:\mathbf{x^a}\subseteq \cdots\subseteq I_{k-d}:\mathbf{x^a}$ and $J_{d+1}:\mathbf{y^b}\supseteq \cdots\supseteq J_{k}:\mathbf{y^b}$. Thus, the above equation becomes
\begin{align*}
    (I+J)_{k}:\mathbf{x^ay^b}& =(I_{k-d}:\mathbf{x^a})+(J_{d+1}:\mathbf{y^b}) \\
    &=\p+\q.
\end{align*}
Thus, $\p+\q\in\ass((I+J)_{k})$ and $\vv_{\p+\q}((I+J)_{k})\leq \vv_{\p}(I_{k-d})+\vv_{\q}(J_{d+1})$. Therefore,
\begin{align}\label{eq7}
    \vv_{\p+\q}((I+J)_{k})\leq\min_{\substack{0\leq d<k \\ 
    \p\in\ass(I_{k-d}) \\ \q\in\ass(J_{d+1})}}(\vv_{\p}(I_{k-d})+\vv_{\q}(J_{d+1})).
\end{align}
Hence, the result follows by inequalities \eqref{eq6} and \eqref{eq7}.
\end{proof}

In literature, the $k$-th symbolic power of an ideal $L\subset R$ is defined in two ways:

 \begin{definition}[Symbolic power]\label{def-symbolic}
 {\rm Let $L\subset R$ be an ideal. Then, the  $k$-th symbolic power of $L$ is defined as
 \begin{enumerate}
     \item $\displaystyle L^{(k)}=\bigcap_{\p\in \mathrm{Ass}(L)}(L^k R_{\p} \cap R).$
     \item $\displaystyle L^{(k)}=\bigcap_{\p\in \mathrm{Min}(L)}(L^k R_{\p} \cap R),$  where $\mathrm{Min}(L)$ is the set of all minimal primes of $L$.
 \end{enumerate}}
 \end{definition}

 The following result regarding the $\vv$-function of symbolic powers of the sum of monomial ideals holds for both definitions of symbolic powers as given in \Cref{def-symbolic}.

\begin{corollary}\label{corsympower}
    Let $I\subset A$ and $J\subset B$ be two monomial ideals. If $\p\subset A$ and $\q\subset B$ be two monomial prime ideals such that $\p+\q\in\ass(I+J)^{(k)}$, then 
    $$\vv_{\p+\q}((I+J)^{(k)})=\min_{\substack{0\leq d<k \\ 
    \p\in\ass(I^{(k-d)}) \\ \q\in\ass(J^{(d+1)})}}(\vv_{\p}(I^{(k-d)})+\vv_{\q}(J^{(d+1)})).$$
    \end{corollary}

\begin{proof}
    Due to the articles \cite{hntt20symbsum} and \cite{hjkn23symbsumgen}, both the definitions of symbolic power given in \Cref{def-symbolic} satisfy the binomial expansion. Again, it is well-known that for a monomial ideal $L$, there exists monomial ideals $Q_1,\ldots,Q_t$ such that for all $h\geq 1$, the $h$-th symbolic power of $L$ can be written as $L^{(h)}=Q_{1}^{h}\cap \cdots\cap Q_{t}^{h}$ (see \cite[Page 2]{hht07}). Now, suppose $L^{(h)}:f=\p$ for some monomial $f$ and some $\p\in\ass(L^{(h)})$. Choose a variable $x\in\p$. Then $xf\in L^{(h)}$, which imply $xf\in Q_{i}^{h}$ for all $1\leq i\leq t$. Then there exist monomials $u_1,\ldots,u_h\in Q_{i}$ such that $xf=u_1\cdots u_h$, and thus, $x$ divides $u_{i}$ for some $i\in\{1,\ldots, h\}$. Consequently, it follows that $u_1\cdots \hat{u_{i}}\cdots u_{h}$ divides $f$. Therefore, we have $f\in Q_{i}^{h-1}$. This is true for each $i\in\{1,\ldots,t\}$, and hence, $f\in Q_{1}^{h-1}\cap \cdots\cap Q_{t}^{h-1}=L^{(h-1)}$. This completes the proof by \Cref{thmmain}.
\end{proof}

Now, we will show that a similar formula of $\vv$-function holds for integral closure filtration whenever there is a binomial expansion.

\begin{definition}{\rm
    Let $I\subset R$ be an ideal. An element $f\in R$ is said to be \textit{integral} over $I$ if there exists an integer $n$ and elements $r_i\in I^{i}$, $i\in\{1,\ldots,n\}$, such that
    $$f^n+r_1f^{n-1}+r_2f^{n-2}+\cdots+r_{n-1}f+r_n=0.$$
The set of all elements of $R$ that are integral over $I$ is called the \textit{integral closure} of $I$, which is denoted by $\overline{I}$.
}
\end{definition}

The integral closure of powers of a graded ideal $I\subset R$ forms a graded $I$-filtration. The following result is concerned on the (local) $\vv$-number of integral closure of powers of the sum of two monomial ideals in two polynomial rings with disjoint variables.

\begin{corollary}\label{corintclosure}
    Let $I\subset A$ and $J\subset B$ be two monomial ideals such that $\overline{(I+J)^k}=\sum_{i+j=k}\overline{I^i}\,\, \overline{J^j}$. If $\p\subset A$ and $\q\subset B$ be two monomial prime ideals such that $\p+\q\in\ass\overline{(I+J)^{k}}$, then 
    $$\vv_{\p+\q}(\overline{(I+J)^{k}})=\min_{\substack{0\leq d<k \\ 
    \p\in\ass(\overline{I^{k-d}}) \\ \q\in\ass(\overline{J^{d+1}})}}(\vv_{\p}(\overline{I^{k-d}})+\vv_{\q}(\overline{J^{d+1}})).$$
\end{corollary}

\begin{proof}
    For a monomial ideal $L$, the integral closure of $L$ can be defined as follows (see \cite[Proposition 12.1.2]{vil15})
    $$\overline{L}:=\langle \mathbf{x^a}\mid (\mathbf{x^a})^m\in L^m \text{ for some } m\geq 1\rangle.$$
    Let us choose any $\p\in \ass(\overline{I^h})$ and a monomial $f\in A$ such that $\overline{I^h}:f=\p$, where $h\geq 1$. Take a variable $x\in \p$. Then $xf\in \overline{I^h}$, which implies $(xf)^l\in (I^h)^l=I^{lh}$ for some $l\geq 1$. Thus, there exist monomials $u_1,\ldots,u_{lh}\in I$ such that $(xf)^l=x^lf^l=u_1\cdots u_{lh}$. Note that $x^l$ divides $u_1\cdots u_{lh}$, and so, without loss of generality, we may assume that $x^l$ divides $u_1\cdots u_d$ for some $d\leq l$. Then we can write $$f^l=(\frac{u_1\cdots u_d}{x^l})\cdot u_{d+1}\cdots u_{l+1}\cdots u_{lh}.$$
    Now, it is easy to see that $u_{l+1}\cdots u_{lh}\in I^{lh-l}=(I^{h-1})^l$. Since $u_{l+1}\cdots u_{lh}$ divides $f^l$, we have $f^l\in (I^{h-1})^l$. Thus, $f\in\overline{I^{h-1}}$ by the above definition of integral closure of monomial ideals. Similarly, if $\q\in \ass(\overline{J^{h}})$ and $g\in B$ be a monomial such that $\overline{J^{h}}:g=\q$, then $g\in \overline{J^{h-1}}$. Hence, the result follows by \Cref{thmmain}. 
\end{proof}

\begin{remark}{\rm 
    Let $I\subset A$ and $J\subset B$ be two monomial ideals. In \cite[Corollary 2.6]{bha23}, a sufficient condition is given for the equality $\overline{(I+J)^k}=\sum_{i+j=k}\overline{I^i}\,\, \overline{J^j}$. As an application, it also has been shown in \cite{bha23} that if (i) $I$ is square-free and $\overline{I^{k}}=I^{(k)}$ for all $k\geq 1$; or (ii) $I$ is a normally torsion-free square-free monomial ideal (i.e. $\ass(I^k)\subset \ass(I)$ for all $k\geq 1$), then the binomial expansion $\overline{(I+J)^k}=\sum_{i+j=k}\overline{I^i}\,\, \overline{J^j}$ holds. Thus, in these cases, we can apply the formula for the $\vv$-number given in \Cref{corintclosure}.
    }
\end{remark}

\noindent \textbf{Acknowledgements:} First, the author would like to thank Dr. Antonino Ficarra for his valuable suggestion. Also, the author thanks the National Board for Higher Mathematics (India) for the financial support through the NBHM Postdoctoral Fellowship and thanks to Chennai Mathematical Institute for providing a good research environment. An Infosys Foundation fellowship partially supports the author.

\printbibliography
\end{document}